\documentclass{article}

\usepackage{hyperref}
\usepackage{cite}
\usepackage{amsfonts}
\usepackage{amsmath} 
\usepackage{amssymb}
\usepackage{amsthm}
\usepackage{enumitem}
\usepackage{graphicx}
\usepackage[all]{xy}
\usepackage{todonotes}
\usepackage{tikz,pgf}

\usetikzlibrary{calc}
\usetikzlibrary{decorations.markings}
\usetikzlibrary{shapes.geometric}
\usetikzlibrary{patterns}
\usetikzlibrary{intersections}

\tikzstyle{vertex}=[circle, draw, fill=black, inner sep=0pt, minimum size=4pt]
\tikzstyle{edge}=[line width=1.5pt,black!50!white]
\tikzstyle{gridp}=[inner sep=1pt,circle,fill=black!70!white]
\tikzstyle{gridl}=[black!50!white]

\tikzstyle{lnode}=[circle,white,draw=black!60!white,fill=black!60!white,inner sep=1pt]
\tikzstyle{llnode}=[circle,white,draw=black!60!white,fill=black!60!white,inner sep=2pt]
\tikzstyle{cnode}=[circle,draw=black!60!white,fill=black!60!white,inner sep=1.5pt]
\tikzstyle{redge}=[edge,red]
\tikzstyle{bedge}=[edge,blue]

\colorlet{ncol}{green!60!black}
\tikzstyle{nvertex}=[vertex, draw=ncol, fill=ncol]
\tikzstyle{edgeq}=[edge,gray!60,densely dashed]
\tikzstyle{nedge}=[edge,ncol]
\tikzstyle{oedge}=[edge,red!60!black]

\theoremstyle{plain}
\newtheorem{lemma}{Lemma}
\newtheorem{proposition}[lemma]{\textbf{Proposition}}
\newtheorem{theorem}[lemma]{\textbf{Theorem}}
\newtheorem{corollary}[lemma]{\textbf{Corollary}}

\theoremstyle{definition}
\newtheorem{definition}[lemma]{\textbf{Definition}}
\newtheorem{remark}[lemma]{\textbf{Remark}}
\newtheorem{example}[lemma]{\textbf{Example}}

\makeatletter
\newcommand*{\da@rightarrow}{\mathchar"0\hexnumber@\symAMSa 4B }
\newcommand*{\da@leftarrow}{\mathchar"0\hexnumber@\symAMSa 4C }
\newcommand*{\xdashrightarrow}[2][]{%
  \mathrel{%
    \mathpalette{\da@xarrow{#1}{#2}{}\da@rightarrow{\,}{}}{}%
  }%
}
\newcommand{\xdashleftarrow}[2][]{%
  \mathrel{%
    \mathpalette{\da@xarrow{#1}{#2}\da@leftarrow{}{}{\,}}{}%
  }%
}
\newcommand*{\da@xarrow}[7]{%
  \sbox0{$\ifx#7\scriptstyle\scriptscriptstyle\else\scriptstyle\fi#5#1#6\m@th$}%
  \sbox2{$\ifx#7\scriptstyle\scriptscriptstyle\else\scriptstyle\fi#5#2#6\m@th$}%
  \sbox4{$#7\dabar@\m@th$}%
  \dimen@=\wd0 %
  \ifdim\wd2 >\dimen@
    \dimen@=\wd2 %
  \fi
  \count@=2 %
  \def\da@bars{\dabar@\dabar@}%
  \@whiledim\count@\wd4<\dimen@\do{%
    \advance\count@\@ne
    \expandafter\def\expandafter\da@bars\expandafter{%
      \da@bars
      \dabar@ 
    }%
  }%
  \mathrel{#3}%
  \mathrel{%
    \mathop{\da@bars}\limits
    \ifx\\#1\\%
    \else
      _{\copy0}%
    \fi
    \ifx\\#2\\%
    \else
      ^{\copy2}%
    \fi
  }%
  \mathrel{#4}%
}
\makeatother

\title{Five Equivalent Representations of a Phylogenetic Tree}

\author{%
Jiayue Qi
 \thanks{Johannes Kepler University Linz, Doctoral Program 
``Computational Mathematics'' (W1214). } \,%
and
Josef Schicho 
\thanks{Johannes Kepler University Linz, RISC.} 
}

\begin{document}
\maketitle

\begin{abstract}
 A phylogenetic tree is a tree with a fixed set of leaves that has no vertices of degree two.
 In this paper, we axiomatically define four other discrete structures on the set of leaves. 
 We prove that each of these structures is an equivalent representation of a phylogenetic tree.
\end{abstract}

\section*{Introduction}

This paper is concerned with the explanation and proof of the following result.

\begin{theorem}\label{thm:intro}
Let $n\ge 3$ be a natural number. The following sets are in one-to-one correspondence to each other.

\begin{enumerate}
\item The set of trees without vertices of degree two and leaves labelled from $1$ to $n$ (also known as ``phylogenetic trees'').
\item The set of set-partitions of $\{1,\dots,n\}$ into at least three subsets such that every singleton (cardinality-one subset) 
occurs in one partition, every subset occurs in at most one partition, and the set of non-singleton subsets is closed under taking the complement.
\item The set of cuts of $\{1,\dots,n\}$, i.e., set-partitions into two subsets $A$, $B$, such that for any two cuts $(A_1,B_1)$ and $(A_2,B_2)$, 
at least one (actually exactly one) of the four sets $A_1\cap A_2$, $A_1\cap B_2$, $B_1\cap A_2$, $B_1\cap B_2$ is empty.
\item The sets of pairs of disjoint two-element subsets of $\{1,\dots,n\}$ that fulfill axioms (X1), (X2), and (X3) in Section~\ref{ss:x}.
\item The set of equivalence relations of three-element subsets of $\{1,\dots,n\}$ whose equivalence classes fulfill axioms (D1) and (D2) in 
Section~\ref{ss:3}.
\end{enumerate}
\end{theorem}

\begin{figure}
\begin{center}
\begin{tikzpicture}[scale=1.2]
  \node[lnode] (1) at (-1.5, -1.) {1};
  \node[lnode] (2) at (-1.5, 1.) {2};
  \node[lnode] (3) at (0., 1.) {3};
  \node[lnode] (4) at (1.5, 1.) {4};
  \node[lnode] (5) at (1.5, -1.) {5};
  \node[cnode] (a) at (-1., 0.) {};
  \node[cnode] (b) at (0., 0.) {};
  \node[cnode] (c) at (1., 0.) {};
  \draw[edge,black] (a)edge(b) (b)edge(c);
  \draw[edge,black] (1)edge(a) (2)edge(a) (3)edge(b) (4)edge(c) (5)edge(c);
  \draw[blue, thick] (-0.7,-1.4) -- (-0.7,1.9);
  \node[blue] at (-0.95,1.7) {$12$};
  \node[blue] at (-0.40,1.7) {$345$};
  \draw[blue, thick] (0.7,-1.4) -- (0.7,1.7);
  \node[blue] at (0.4,1.5) {$123$};
  \node[blue] at (0.95,1.5) {$45$};
\end{tikzpicture}
\hspace{1cm}
%
\begin{tikzpicture}[scale=1.5]
  \node[lnode] (1) at (-0.5, -1) {1};
  \node[lnode] (2) at (-0.5, 1) {2};
  \node[lnode] (3) at (1, 1) {3};
  \node[lnode] (4) at (2.5, 1) {4};
  \node[lnode] (5) at (2.5, -1) {5};
  \node[cnode] (a) at (0, 0) {};
  \node[cnode] (b) at (1, 0) {};
  \node[cnode] (c) at (2, 0) {};
  \draw[edge,black] (a)edge(b) (b)edge(c);
  \draw[edge,black] (1)edge(a) (2)edge(a) (3)edge(b) (4)edge(c) (5)edge(c);
  
  \draw[red, thick] (1,0) circle (.45);
  \node[blue, thick] at (0.9,0.3) {3};
  \node[blue, thick] at (0.7,-0.123) {12};
  \node[blue, thick] at (1.3,-0.123) {45};
  
  \draw[red, thick] (0,0) circle (.45);
  \node[blue, thick] at (-0.01,0.222) {2};
  \node[blue, thick] at (-0.2,-0.2) {1};
  \node[blue, thick] at (0.211,-0.123) {345};
  
   \draw[red, thick] (2,0) circle (.45);
  \node[blue, thick] at (2.01,0.222) {4};
  \node[blue, thick] at (2.2,-0.2) {5};
  \node[blue, thick] at (1.789,-0.123) {123};
\end{tikzpicture}
\end{center}
\caption{Starting from a tree without vertices of degree three, we can obtain a set of cuts 
(i.e. set-partitions into two subsets) of its set of leaves (left subfigure), and a set of set-partitions into at least 
three subsets (right subfigure).}
\label{fig:5tree}
\end{figure}
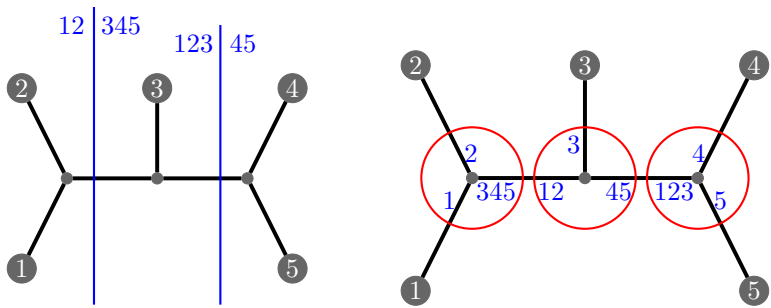

Figure~\ref{fig:5tree} shows the construction of a set of cuts and the construction of a set of set-partitions into at least three subsets, for a fixed phylogenetic tree. To obtain a set of pairs of disjoint two-element subsets, take a set of cuts, 
and pick, for each pair, a two element subset on the left and a two-element subset on the right of one of the cuts. To construct a set of equivalence relations
of three-element subsets, we recall the following property of trees: for any three distinct leaves, there is a unique vertex such that the three paths from this vertex to the three leaves are edge-disjoint (see Lemma~\ref{lem:unique_vertex}).
There is a natural way to use this fact in order to define an equivalence relation on three-element subsets: two three-element subsets 
are equivalent if and only if the unique vertex is the same for the two triples.
For larger $n$, it is not so clear that all these constructions are in bijection, i.e., 
the phylogenetic tree can be uniquely recovered from its set of cuts, or from any of the other three structures. 
We will prove it in full detail. The proofs are combinatoric, but not all trivial.

We offer an informal taxonomy of the four equivalent representations of phylogenetic trees. 
The collections of partitions and the sets of cuts are 
{\em macroscopic} pictures, in the sense that they are composed of elements of bigger scales.
In contrast, the equivalence relations
on three-element-subsets and the crossing relations are composed of smaller-scale elements; 
they are {\em microscopic} pictures: 
the crossing relations are just quaternary relations
on the set of leaves, and the equivalence relation on three-element subsets can be considered as 6-ary relations on the set of leaves.
The partitions and the equivalence relations focus on the vertices of the phylogenetic tree, in particular the non-leaves, while the cuts and the
crossing relations focus on its edges.

The relevance of these four characterizations of phylogenetic trees lies in their applications 
to the construction of the Knudsen-Mumford moduli space of $n$-marked curves of genus zero. 
In fact, they allow us to give a very explicit and even elementary construction of this moduli 
space (in a joint research in progress together with H. Hauser), in contrast to the high-flown machinery 
in the original papers \cite{KnudsenMumford1976,Knudsen1983} and to simpler, but still quite algebraic 
constructions in \cite{Keel1992,Kapranov:93,Keel_Tevelev:09,Monin_Rana:17}. Connections between
moduli space and phylogenetic trees have also been observed in \cite{GibneyMaclagan2010,Luca}, and
the cuts induced by a phylogenetic tree also play a role  in 
\cite{SturmfelsSullivant08,GibneyMaclagan2010}.

This paper has two sections. In the first section, we give the definitions of these five representations and some constructions 
from one to another. The second section contains the more involved constructions and the proof that 
all constructions are in bijection.

\section{Structures and Axioms} \label{sec:axioms}


In this section, we axiomatically define five discrete structures on a fixed finite set $N$ of cardinality at least 3:
phylogenetic trees, collections of partitions, sets of cuts, crossing relations, and equivalence relations of triples.
We will also introduce some functions converting one of these structures to another,
and use them to construct examples. Other (more involved) conversion functions will be introduced in Section~\ref{sec:conversions}.

\subsection{Trees}

Recall that an unrooted tree is an undirected graph that is connected and has no cycles.

\begin{definition}
A {\em phylogenetic tree} with leaf set $N$ is an unrooted tree $(V,E)$ without vertex of degree~2
such that $N\subset V$ is the set of leaves. We say that two phylogenetic trees $(V_1,E_1)$ and
$(V_2,E_2)$ with leaf set $V$ are isomorphic if and only if there is a graph isomorphism
that restricts to the identity on the subset of leaves.

The set ${\cal T}_N$ is defined as the set of all isomorphism classes of phylogenetic trees
with leaf set $N$.
\end{definition}

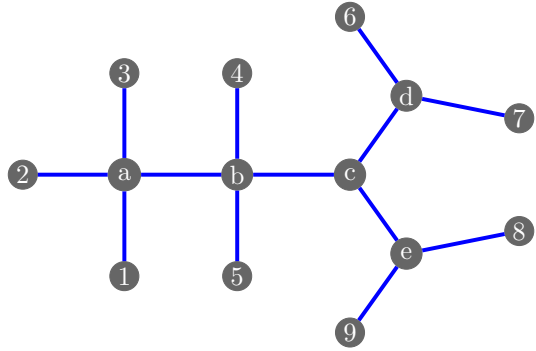
\begin{figure}
        \begin{center}
                \begin{tikzpicture}[scale=1.5]
                        \node[lnode] (1) at (1., -0.9) {1};
                        \node[lnode] (2) at (0.1, 0.) {2};
                        \node[lnode] (3) at (1., 0.9) {3};
                        \node[lnode] (4) at (2., 0.9) {4};
                        \node[lnode] (5) at (2., -0.9) {5};
                        \node[lnode] (6) at (3., 1.4) {6};
                        \node[lnode] (7) at (4.5, 0.5) {7};
                        \node[lnode] (8) at (4.5, -0.5) {8};
                        \node[lnode] (9) at (3., -1.4) {9};
                        \node[llnode] (A) at (1., 0.) {a};
                        \node[lnode] (B) at (2., 0.) {b};
                        \node[llnode] (C) at (3., 0.) {c};
                        \node[lnode] (D) at (3.5, 0.7) {d};
                        \node[llnode] (E) at (3.5, -0.7) {e};
                        \draw[bedge] (1)edge(A) (2)edge(A) (3)edge(A) (4)edge(B) (5)edge(B) (6)edge(D) (7)edge(D) (8)edge(E) (9)edge(E);
                        \draw[bedge] (A)edge(B) (B)edge(C) (C)edge(D) (C)edge(E);
                \end{tikzpicture}
        \end{center}
\caption{This is a phylogenetic tree with leaf set $N=\{1,\dots,9\}$. The set of internal nodes is $V\setminus N=\{a,b,c,d,e\}$. }
\label{fig:tree}
\end{figure}

An example of a phylogenetic tree with leaf set $N=\{1,\dots,9\}$ can be seen in Figure~\ref{fig:tree}.

\subsection{Sets of Partitions}

Recall that a partition of $N$ is a set of disjoint and non-empty subset of $N$ such that their union
is $N$.

\begin{definition}
A collection/set of partitions of $N$ is {\em phylogenetic} if it fulfills the following axioms:
\begin{description}
\item[(P1)] Every partition has at least 3 subsets; we also call these subsets the {\em parts}.
\item[(P2)] Every one-element subset of $N$ is a part of some partition.
\item[(P3)] Every subset of $N$ is a part of at most one partition.
\item[(P4)] For every part $A\subset N$ of cardinality bigger than one, the complement $N\setminus A$
  is also a part (necessarily of a different partition, by Axiom~(P1)).
\end{description}
We denote as ${\cal P}_N$ the set of all phylogenetic collections of partitions of $N$.
The set $\overline{{\cal P}_N}$ is the set of all sets of collections of partitions of $N$.
\end{definition}

\begin{example}
\label{ex:parts}
For $N=\{1,\dots,9\}$, let $P=\{ p_a,p_b,p_c,p_d,p_e\}$ be the collection of the partitions
\[ p_a = \{ \{1\}, \{2\}, \{3\}, \{4,5,6,7,8,9\} \} ,\]
\[ p_b = \{ \{1,2,3\}, \{4\},  \{5\}, \{6,7,8,9\} \} , \]
\[ p_c = \{ \{1,2,3,4,5\}, \{6,7\}, \{8,9\} \} ,\]
\[ p_d = \{ \{1,2,3,4,5,8,9\}, \{6\}, \{7\} \} , \]
\[ p_e = \{ \{1,2,3,4,5,6,7\}, \{8\}, \{9\} \} . \]
It can be checked that the axioms (P1),(P2),(P3),(P4) are fulfilled. Therefore the collection $P$ is phylogenetic.
\end{example}

The above example can be constructed from the tree in Figure~\ref{fig:tree} in a systematic way, which is described in
the following definition.

\begin{definition}
For every phylogenetic tree $T=(V,E)$ with leaf set $N$, we define a collection of partitions
which is in bijection with the set $V\setminus N$ of non-leaves, as follows.

For each non-leaf vertex 
$v$, and for each edge $e$ incident to~$v$, 
we have a subset of leaves containing the leaves $i$ such that the unique path from $v$ to $i$ begins with $e$.
For each non-leaf vertex $v$, these subsets are the parts of a partition of $N$.
The collection of partitions of $N$ is denoted by $P_T$.
The function ${\cal T}_N\to \overline{{\cal P}_N}$ mapping the class of $T$ to $P_T$ is denoted by $t_{TP}$ and we call it
the transformation from trees to partition collections.
\end{definition}

\begin{proposition}
For every phylogenetic tree $T=(V,E)$ with leaf set $N$, the collection $P_T$ is phylogenetic. 
\end{proposition}

\begin{proof}
Every non-leaf has at least 3 edges, hence (P1) holds.
Every leaf has a unique neighbor
which must be a non-leaf, otherwise the tree would only have two vertices (which violates our assumption on the cardinality of $N$); this implies (P2). 
Uniqueness of the neighbor
also implies (P3) for the special case of cardinality-one parts. 

Let $A$ be a part of $P_T$ such that $2\le |A|\le |N|-2$.  Then there are non-leaves $a$, $b$
and an edge $e=\{a,b\}\in E$ such that $A$ is equal to the set of leaves $i$ such that the unique path from $a$ to $i$ contains $e$.
And then $N\setminus A$ is the set of leaves $j$ such that the unique path from $b$ to $j$ contains $e$, hence $N\setminus A$
is part of the partition corresponding to the non-leaf $b$. This shows (P4).

It remains to prove Axiom (P3). Suppose that $A$ belongs to two distinct partitions of $P_T$,
corresponding to non-leafs $a$ and $a'$ respectively. Let $e=\{a,b\}$ and $e'=\{a',b'\}$ be
the two edges corresponding to $A$ in the partitions corresponding to $a$ and $a'$ respectively.
There is a unique path $p_{a,a'}$ between $a$ and $a'$. 
In the sequel, we do case distinctions on whether $e$ and $e'$ belongs to path $p_{a,a'}$ or not.
\begin{enumerate}
 \item Neither $e$ nor $e'$ belongs to $p_{a,a'}$. In this case, we remove the edge $e$, and we obtain two components 
 $T_b$ and $T_a$, where $T_b$ contains the elements
of $A$ and $T_a$ contains the elements of $N\setminus A$. Now we remove $e'$ in $T_a$, and we further obtain 
two components $T_{b'}$ and $T_{a'}$, where $T_{b'}$ contains the elements from $A$. However,
$T_{b'}$ and $T_b$ are distinct components, they cannot both contain the elements from $A$. This is a contradiction.
\item Both $e$ and $e'$ belong to $p_{a,a'}$. In this case, we can argue analogously, by interchanging the roles
of $a$ with $b$, and the roles of $a'$ with $b'$. We remove edge $e$, obtaining components $T_a$ which contains
elements in $N\setminus A$, and $T_b$ which contains elements in $A$. Similarly, there must be a unique path from
$b$ to $b'$ connecting edge $e$ and $e'$. Hence, $e'$ is in $T_b$. Now, in component $T_b$, we remove the edge $e'$, 
obtaining components $T_{a'}$ and $T_{b'}$, where $T_{b'}$ is the component containing elements in $N\setminus A$.
However, $T_{b'}$ and $T_a$ are distinct components, hence cannot both contain elements in $N\setminus A$. This 
is a contradiction.
\item Only one of $e$ and $e'$ belongs to $p_{a,a'}$. Without loss of generality, assume that $e'$ belongs to path $p_{a,a'}$ and $e$ does not.
We remove edge $e$ from the tree $T$, obtaining $T_b$ and $T_a$, where the set of leaves in $T_b$ intersected with $N$
is $A$. We remove edge $e'$ from the tree $T$, obtaining $T_{b'}$ and $T_{a'}$, 
where the set of leaves in $T_{b'}$ intersected with $N$ is $A$. Because $T$ is phylogenetic, there
is an edge $e''$ not contained in the path $p_{a,a'}$, but incident with some vertex on this path. 
Following this edge, we eventually
arrive at some leaf $l$. Then we have that $l$ is in $T_{b'}\cap N$ but not in 
$T_b\cap N$. This is a contradiction.
\end{enumerate}

Hence, Axiom (P3) holds.
\end{proof}

\subsection{Sets of Cuts}

In this structure, we are particularly interested in partitions that violate (P1).

\begin{definition}
A {\em cut} of $N$ is a partition of $N$ into two subsets $A,B$ of cardinality larger than one. 
The subsets $A$, $B$ are called {\em clusters}. We denote such a cut as $(A\mid B)=(B\mid A)$, and omit the
curly brackets when $A$ and $B$ are given by the enumeration of elements.
Axiom for a set $C$ of cuts of $N$ to be {\itshape phylogenetic} is as follows. And we denote as $cl(C)$ 
the set of all clusters of $C$.
\begin{description}
\item[(C)] For any  two cuts $(A_1\mid B_1)$, $(A_2\mid B_2)$ in $C$, at least one of the following four
  sets is empty: $A_1\cap A_2$, $A_1\cap B_2$, $B_1\cap A_2$, $B_1\cap B_2$.
\end{description}
Denote as ${\cal C}_N$ the set of all phylogenetic sets of cuts of $N$.
The set $\overline{{\cal C}_N}$ is the set of all sets of cuts of $N$.
\end{definition}
\begin{remark}
 One can check that actually we can omit ``at least'' in the above statement, since when it holds, it cannot happen that two 
 of those four sets are both empty.
\end{remark}

\begin{example}
\label{ex:cuts}
For $N=\{1,\dots,9\}$, let $C=\{ c_x,c_y,c_z,c_w\}$ be the set of cuts
\[ c_x = (1,2,3\mid 4,5,6,7,8,9), \]
\[ c_y = (1,2,3,4,5\mid 6,7,8,9), \]
\[ c_z = (1,2,3,4,5,8,9\mid 6,7), \]
\[ c_w = (1,2,3,4,5,6,7\mid 8,9). \]
It can be checked that axiom (C) is fulfilled, hence $C$ is phylogenetic.

Note that the clusters in these cuts are exactly the parts that appear in some partition 
in Example~\ref{ex:parts} that have cardinality bigger than one. Every cut corresponds to an internal edge of 
the phylogenetic tree, i.e., an edge connecting two non-leafs.
The clusters are just the set of leaves of the two connected components which arise when the corresponding edge is removed.
\end{example}

\begin{definition}
Let $T$ be a phylogenetic tree with leaf set $N$. Then $C_T$ is the set of cuts of $N$ that corresponds to 
internal edges of $T$,
with clusters being the two sets of leaves of the two components that arise when the corresponding edge is removed.
The function ${\cal T}_N\to \overline{{\cal C}_N}$, $[T]\mapsto C_T$ is denoted by $t_{TC}$ and we call it
the transformation from trees to sets of cuts.

Let $P$ be a phylogenetic collection of partitions. The set of cuts whose clusters are exactly the parts of cardinality 
at least~2 is denoted by $C_P$.
The function ${\cal P}_N\to \overline{{\cal C}_N}$, $P\mapsto C_P$ is denoted by $t_{PC}$ and we call it
the transformation from partition collections to sets of cuts.
\end{definition}

\begin{proposition}
For every tree $T$, we have $C_{P_T}=C_T$; in other words, we have $t_{PC}\circ t_{TP}=t_{TC}$.
\end{proposition}

\begin{proof}
Straightforward.
\end{proof}

\subsection{Crossing Relations} \label{ss:x}

\begin{definition} \label{def:x}
A {\em crossing relation} is a set $X$ of unordered pairs of disjoint cardinality-two subsets of $N$. We write
its element as $(i,j\mid k,l)$ - such that if $(i,j\mid k,l)\in X$, then the information that
$i,j,k,l$ are pairwise distinct is contained. And we call it a {\itshape cross (of $X$)},
since we can interchange $i,j$ or $k,l$, or two sets $\{i,j\}$ with $\{k,l\}$ without 
changing the element. 

Axioms for a crossing relation $X$ to be {\itshape phylogenetic} are as follows.
\begin{description}
\item[(X1)] If $(i,j\mid k,l)$, then not $(i,k\mid j,l)$, i.e., $(i,k\mid j,l)\notin X$.
\item[(X2)] If $(i,j\mid k,l)$ and $(i,j\mid k,m)$ and $l\ne m$, then $(i,j\mid l,m)$.
\item[(X3)] If $(i,j\mid k,l)$ and $m$ is distinct from $i,j,k,l$, then $(i,j\mid k,m)$ or $(i,m\mid k,l)$. Note that this ``or'' here means at least
one should hold and it may happen that both hold.
\end{description}
Denote as ${\cal X}_N$ the set of all phylogenetic crossing relations on $N$.
Denote as $\overline{{\cal X}_N}$ the set of all crossing relations of $N$.
\end{definition}

\begin{example} 
\label{ex:cross}
Let $N:=\{1,\dots,9\}$. We define a crossing relation as follows.
For any $i,j,k,l$ that are pairwise distinct, the relation $(i,j\mid k,l)$ holds if and only
if one of the following statements is true:
\begin{itemize}
  \item $i,j\in \{1,2,3\}$ and $k,l\in \{4,5,6,7,8,9\}$ ($45$ crosses);
  \item $i,j\in \{1,2,3,4,5\}$ and $k,l\in \{6,7,8,9\}$ ($60$ crosses);
  \item $i,j\in \{1,2,3,4,5,8,9\}$ and $\{k,l\}=\{6,7\}$ ($21$ crosses);
  \item $i,j\in \{1,2,3,4,5,6,7\}$ and $\{k,l\}=\{8,9\}$ ($21$ crosses).
\end{itemize}
(Its relation with Example~\ref{ex:cuts} is apparent.) In total, this crossing relation
consists of $108$ crosses. We will see later that this crossing
relation is phylogenetic.
\end{example}

The following definition is a generalization of the construction in Example \ref{ex:cross}.

\begin{definition}
For every set $C$ of cuts, we define a crossing relation $X_C$ as follows:
$(i,j,k,l)$ is in $X_C$ if and only if $C$ contains a cut $(A\mid B)$ such that $i,j\in A$ and $k,l\in B$.
The function ${\cal C}_N\to \overline{{\cal X}_N}$, $C\mapsto X_C$ is denoted by $t_{CX}$, we call it
the transformation from sets of cuts to crossing relations.
\end{definition}

\begin{remark}
The moduli space of $n$-pointed stable curves of genus zero has a natural decomposition into strata that
correspond to phylogenetic trees with leaf set of cardinality $n$. 

For any such stratum $T$, the crossing
relation is then exactly the set of $(i,j\mid k,l)$ such that the cross ratio of the four marked points
$p_i,p_j,p_k,p_l$ has value~1. As we will see, it is possible to transform a phylogenetic crossing relation
to a phylogenetic tree. In the context of moduli spaces, this is equivalent to saying that we can recover
the dual graph of the $n$-pointed graph from the values of its cross ratios.
\end{remark}

\subsection{Equivalences of Triples}
\label{ss:3}

A {\em triple} in $N$ is a 3-element subset of $N$. We denote the set of triples in $N$ by ${N\choose 3}$. 
A set $S\subset{N\choose 3}$ of triples is called {\em diverse} if it is non-empty and it fulfills the following two axioms:
\begin{description}
\item[(D1)] If $\{i,j,k\}\in S$, and $l\in N$, then $S$ also contains one of the triples $\{ i,j,l\}$, $\{ i,k,l\}$, or $\{j,k,l\}$.
\item[(D2)] Let $a,b,c,x,y,z\in N$. If $S$ contains the triples $\{ a,x,y\}$, $\{ b,y,z\}$, and $\{ c,x,z\}$, then it also contains $\{ x,y,z\}$.
\end{description}

We say that an equivalence relation on ${N\choose 3}$ is {\itshape phylogenetic} if and only if the following axiom is fulfilled:
\begin{description}
\item[(E0)] Each class of the equivalence relation is diverse.
\end{description}
Denote as ${\cal E}_N$ the set of all phylogenetic equivalence relations on the triples of $N$.
Denote as $\overline{{\cal E}_N}$ the set of all equivalence relations on the triples of $N$.

\begin{example} \label{ex:equivsmall}
Let $N=\{1,2,3,4,5\}$. We define an equivalence relation with three classes as follows:
\[ \{ 1,2,3\}\sim\{ 1,2,4\}\sim\{ 1,2,5\} , \]
\[ \{ 1,4,5\}\sim\{ 2,4,5\}\sim\{ 3,4,5\} , \]
\[ \{ 1,3,4\}\sim\{ 1,3,5\}\sim\{ 2,3,4\}\sim\{ 2,3,5 \}. \]
It can be checked that the axioms (D1) and (D2) are fulfilled in each class, hence the equivalence is 
phylogenetic.
\end{example}


In order to construct interesting equivalence relations of triples, we need a lemma on trees.

\begin{lemma}\label{lem:unique_vertex}
Let $T=(V,E)$ be a tree, and let $i,j,k\in N$ be pairwise distinct leaves.
Then there is a unique vertex $v\in V\setminus N$ such that the three paths
from $v$ to $i$, $j$, and $k$ are
edge-disjoint.
\end{lemma}

\begin{proof}
Let $\pi_{ij}$ be the unique path connecting $i$ and $j$, and let $\pi_{ik}$ be the unique path connecting $i$ and $k$.
The common edges of $\pi_{ij}$ and $\pi_{ik}$ also form a path, which connects $i$ to some non-leaf, and this is the
vertex $v$ we are looking for. Indeed, the edges that are in $\pi_{ij}$ but not in $\pi_{ik}$ connect $j$ with $v$,
and the edges that are in $\pi_{ik}$ but not in $\pi_{ij}$ connect $v$ with $k$. The property of $v$ which is claimed
in the lemma implies that the common edges of $\pi_{ij}$ and $\pi_{ik}$ form a path from $i$ to $v$, and this implies uniqueness.
\end{proof}

\begin{definition}\label{def:tree_to_triples}
Let $T=(V,E)\in {\cal T}_N$ be a phylogenetic tree. We define an equivalence relation $\sim_T$ on ${N\choose 3}$
as follows: $\{ i,j,k\}\sim_T \{l,m,n\}$ holds if and only if
the unique non-leaf $v$ such that the three paths from $v$ to $i$, $j$, and $k$ are edge-disjoint
is equal to
the unique non-leaf $w$ such that the three paths from $w$ to $l$, $m$, and $n$ are edge-disjoint.
The function ${\cal T}_N\to \overline{{\cal E}_N}$, $[T]\mapsto \sim_T$ is denoted by $t_{TE}$, we call it
the transformation from trees to equivalence relations.
\end{definition}

\begin{example} \label{ex:tb}
Let $N=\{1,\dots,9\}$, and let $T=(V,E)$ be the phylogenetic tree in Figure~\ref{fig:tree}. Then the equivalence relation
$\sim_T$ in ${N\choose 3}$ has five equivalence classes $E_a,\dots,E_b$, corresponding to the five non-leaves in $V\setminus N$:
\begin{itemize}
\item The class $E_a$ consists of all triples $\{i,j,k\}$ such that ($i,j\in\{1,2,3\}$ and $k\in\{ 4,5,6,7,8,9\}$) or
	($i=1$, $j=2$, and $k=3$). These
	are 18+1=19 triples.
\item The class $E_b$ consists of all triples $\{i,j,k\}$ such that ($i,j\in\{1,2,3\}$ and $j\in\{ 4,5\}$ and $k\in\{ 6,7,8,9\}$), 
	or ($i=4$ and $j=5$ and $k\in\{ 1,2,3,6,7,8,9\}$). These are $24+7=31$ triples.
\item The class $E_c$ consists of all triples $\{i,j,k\}$ such that $i\in\{ 1,2,3,4,5\}$, $j\in\{ 6,7\}$ and $k\in\{ 8,9\}$.
	These are 20 triples.
\item The class $E_d$ consists of all triples $\{i,6,7\}$ such that $i\in\{ 1,2,3,4,5,8,9\}$. These are 7 triples.
\item The class $E_e$ consists of all triples $\{i,8,9\}$ such that $i\in\{ 1,2,3,4,5,6,7\}$. These are 7 triples.
\end{itemize}
Note that $19+31+20+7+7=84={9\choose 3}$, which indicates that we did not make a mistake --- every triple occurs in exactly one class.

\end{example}

In the next section, we will introduce the transformation between equivalences of triples and crossing relations.

\section{Conversions} \label{sec:conversions}

In this section, we prove that the five structures introduced in Section~\ref{sec:axioms} are equivalent.
In Section~\ref{sec:axioms}, we already introduced the maps shown in Figure~\ref{fig:diag1}, and we have seen that the
triangle contained in this diagram is commutative. We still have to show that the images of $t_{TC}$, $t_{CX}$,
$t_{TE}$ are phylogenetic, we have to construct more conversion maps so that the diagram has
directed paths between any two vertices, and we have to show that the enlarged diagram commutes.

\begin{figure}
\[ \xymatrix{
  & & {\cal T}_N \ar[lld] \ar[rrd] \ar[rdd] & & \\
  \overline{{\cal E}_N} & & & & {\cal P}_N \ar[ld] \\ 
  & \overline{{\cal X}_N} & & \overline{{\cal C}_N} \ar[ll] &
} \]
\caption{This diagram shows the conversion maps between different types of structures that have been defined 
in Section~\ref{sec:axioms}. We also have seen that the triangle on the right part is commutative.}
\label{fig:diag1}
\end{figure}
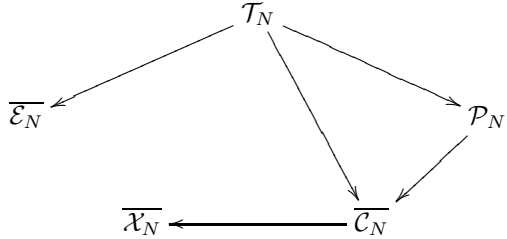

\subsection{Trees and Partitions}

\begin{definition}
For every set $P=\{p_1,\dots,p_m\}$ of partitions of $N$, we define the graph $G_P$ as follows. 
The vertex set is $N\cup P$.
Two vertices in $p,q\in P$ are connected by an edge if and only if there is a cut $(A\mid B)$ such that 
$A\in p$ and $B\in q$.
A vertex $p\in P$ and a vertex $i\in N$ are connected if $\{i\}\in p$.
There is no edge connecting two vertices in $N$.
\end{definition}

In the following, we will show that $G_P$ is a phylogenetic tree whenever $P$ is phylogenetic, and that the
construction $P\to G_P$ is the inverse of $t_{TP}$.

\begin{theorem} \label{thm:hh}
Assume that $P$ is a phylogenetic set of partitions of $N$. Then $G_P$ is a phylogenetic tree.
\end{theorem}

\begin{proof}[Proof (H. Hauser)]
Let $i\in N$ and $p\in P$. We claim that there is a path connecting $i$ and $p$.
Let $A$ be the part in the partition $p$ that contains $i$. If $A=\{i\}$, then there is an edge connecting $i$ and $p$.

If $A$ has cardinality bigger than 1, then there is a unique partition $q$ containing $N\setminus A$. It also has
a unique part $B$ that contains $i$, which must be a strict subset of $A$. By induction
on the cardinality of $A$, there is a path connecting $i$ and $q$. This shows the existence
of a path connecting $i$ and $p$.

It follows that the graph $G_P$ is connected. In order to show that $G_P$ is a tree, it suffices to show that
it has no cycle. The vertices of such a cycle would have to be in $P$, because the vertices in $N$ have degree 1.

Let $(p_1,\dots,p_k,p_{k+1}=p_1)$ be a cycle. For $r=1,\dots,k$, there is a unique part $A_r\in p_r$ that contains $i$.
For the edge $e=p_1p_2$, one part of its corresponding cut, say $(I,J)$ must contain $i$, and it 
must be either $A_1$ or $A_2$. If it is $A_1$, then we have $A_{r+1}\subsetneq A_r$ for $r=1,\cdots, k$ because of
Axiom (P3). If it is $A_2$, then we have that $A_r\subsetneq A_{r+1}$ for $r=1,\cdots, k$. Both cases lead to
$A_1\subsetneq A_1$, which is a contradiction.

The degree of any vertex in $P$ is equal to the number of its parts, which is at least three.
Therefore, the tree $G_P$ is phylogenetic.
\end{proof}

If $T$ is a phylogenetic tree, then it is straightforward to check that $G_{P_T}$ is isomorphic to $T$.
Also, if $P$ is a phylogenetic set of partitions, then $P_{G_P}=P$. Hence the construction $P\to G_P$
provides the inverse to $t_{TP}:{\cal T}_N\to{\cal P}_N$.

\subsection{Trees and Cuts}

\begin{proposition}
For every phylogenetic tree $T=(V,E)$ with leaf set $N\subset V$, the set $C_T$ of cuts is phylogenetic.
\end{proposition}

\begin{proof}
Let $(A_1\mid B_1)$ and $(A_2\mid B_2)$ be two arbitrary cuts in $C_T$, and let $e_1$ and $e_2$ be their corresponding edges. 
If we remove both edges from the graph, then we get at most three components. correspondingly, we obtain three leaf sets. Each of the four sets $A_1\cap A_2$, 
$A_1\cap B_2$, $A_2\cap B_1$, $A_2\cap B_2$ equals to one of these leaf sets, if not empty. Also, note that these four sets are pairwise disjoint. Therefore,
at least one of these four sets must be empty. Since the two cuts were chosen arbitrarily, it follows that Axiom~(C) is fulfilled,
and $C_T$ is phylogenetic.
\end{proof}

For the construction of transformation $t_{CT}$ from cuts to trees, recall the following concept: if $(P,\le)$ is a finite
partially ordered set, then the {\em Hasse diagram} of $(P,\le)$ is the directed graph with vertex set $P$. And there is an
edge from vertex $a$ to vertex $b$ if and only if ($a\leq b$ and for all $c$ such that $a\le c\le b$, 
we have $a=c$ or $b=c$).

We call a set/subset with exactly one element a {\em singleton}.

\begin{definition}\label{def:construction_cuts_tree}
Let $C$ be a phylogenetic set of cuts. Let $c=(A\mid B)$ be a cut in $C$. Let $V_A$ be the set of all clusters or singletons 
that are subsets of $A$ (including $A$ itself). 
Let $G_A=(V_A,E_A)$ be the Hasse diagram of $V_A$ ordered by inclusion.
Let $V_B$ be the set of all clusters or singletons that are subsets of $B$ (including $B$ itself).
Let $G_B=(V_B,E_B)$ be the Hasse diagram of $V_B$ ordered by inclusion.

Let $G_{C,c}=(V_{C,c},E_{C,c})$ be the undirected graph with $V_{C,c}:=V_A\cup V_B$, and $E_{C,c}$ is equal
to the union of $E_A$ and $E_B$ --- forgetting the direction --- plus one extra edge connecting $A$ and $B$.
We call $G_{C,c}$ the {\em cut graph of $c$}. 
\end{definition}

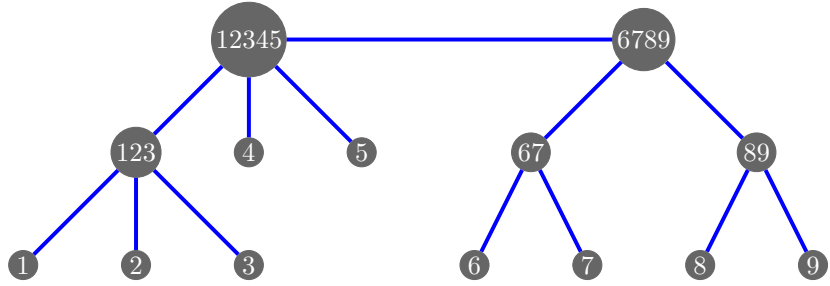
\begin{figure}
        \begin{center}
                \begin{tikzpicture}[scale=1.5]
                        \node[lnode] (1) at (0., 0.) {1};
                        \node[lnode] (2) at (1., 0.) {2};
                        \node[lnode] (3) at (2., 0.) {3};
                        \node[lnode] (4) at (2., 1.) {4};
                        \node[lnode] (5) at (3., 1.) {5};
                        \node[lnode] (6) at (4., 0.) {6};
                        \node[lnode] (7) at (5., 0.) {7};
                        \node[lnode] (8) at (6., 0.) {8};
                        \node[lnode] (9) at (7., 0.) {9};
                        \node[lnode] (A) at (1., 1.) {123};
                        \node[lnode] (B) at (2., 2.) {12345};
                        \node[lnode] (C) at (5.5, 2.) {6789};
                        \node[lnode] (D) at (4.5, 1.) {67};
                        \node[lnode] (E) at (6.5, 1.) {89};
                        \draw[bedge] (1)edge(A) (2)edge(A) (3)edge(A) (4)edge(B) (5)edge(B) (6)edge(D) (7)edge(D) (8)edge(E) (9)edge(E);
                        \draw[bedge] (A)edge(B) (B)edge(C) (C)edge(D) (C)edge(E);
                \end{tikzpicture}
        \end{center}
\caption{For the set of cuts in Example~\ref{ex:cuts}, this figure shows the two Hasse diagrams of the partial orders of clusters
	that are subsets of $A=\{1,2,3,4,5\}$ (on the left) and $B=\{6,7,8,9\}$ (on the right), respectively; and the graph $G_{C_{c_y}}$ which is defined 
	as the union of these two Hasse diagrams plus one extra edge. We omit the directions of the edges in the Hasse diagram here;
	the convention is that they are always upward.}
\label{fig:hasse}
\end{figure}

\begin{example}
Let $N=\{1,\dots,9\}$. Let $C$ be the set of cuts in Example~(\ref{ex:cuts}). Recall that $c_y=(A\mid B)$ where
$A=\{1,2,3,4,5\}$ and $B=\{6,7,8,9\}$. Figure~\ref{fig:hasse} shows the Hasse diagrams $G_A$ and $G_B$ and the
graph $G_{C,c_y}$.
\end{example}


\begin{lemma} \label{lem:cut_to_tree_phylogenetic}
Let $C$ be a phylogenetic set of cuts of $N$, and let $c=(A\mid B)\in C$. Then $G_{C,c}$ is a phylogenetic tree with
leaf set $\{ \{i\}\mid i\in N\}$.
\end{lemma}

\begin{proof}
Since the partially ordered sets $V_A$ and $V_B$ have a largest element, the two Hasse diagrams are connected. The 
extra edge connects the two Hasse diagrams, therefore $G_{C,c}$ is connected.

Let $v,w\in V_A$. By Axiom~(C), the set of clusters $u$ such that $v\le u\le w$ is totally ordered by inclusion.
Therefore, the Hasse diagram $G_A$ has no cycle. The same holds for the Hasse diagram $G_B$. Hence both Hasse diagrams are trees.
$G_{C,c}$ is obtained by connecting two trees with one extra edge, and so $G_{C,c}$ is also a tree. Its leaves are the
minimal elements of the two partial orders, which are exactly the singletons of $N$.

Assume, for the sake of contradiction, that $G_{C,c}$ has a vertex of degree two. Then there is a cluster $D$ of $C$ such that 
the partially ordered set $V_D$ has a ``second largest element'' $D'$, i.e. an element which is largest in 
the subset of elements not equal to $D$. Then $D\setminus D'$ is not empty, hence there is a singleton $\{a\}$
such that $\{a\}\subseteq D$ but $\{a\}\not\subseteq D'$, which is a contradiction.
\end{proof}

\begin{proposition} \label{prop:T2C}
Let $T=(V,E)$ be a phylogenetic tree with leaf set $N\subset V$. Let $e=\{u,v\}\in E$ be an internal edge and 
$c=(A\mid B) \in C_T$ be its corresponding cut. 
Then the phylogenetic tree $G_{C_T,c}=(V_G,E_G)$ is isomorphic to $T$. 
\end{proposition}
\begin{proof}
We obtain two components $T_u$ and $T_v$ after removing 
$e$ from $T$. The set of leaves of $T_u$ is $A\cup\{u\}$, and the set of leaves of $T_v$ is $B\cup\{v\}$. 
We define a map $f:V\to V_G$. Let $w\in V$.
\begin{enumerate}
\item If $w$ is a leaf of $T$, then $f(w):=\{w\}$. 
\item If $w$ is an internal vertex of $T$ contained in $T_u$, then let $e'$ denote the first edge on the unique path in $T$ from $w$ to $v$. 
Let $(A'\mid B')$ be the corresponding cut, and assume without loss of generality that $A'\subset A$.
We set $f(w):=A'$.
\item Analogously, if $w$ is an internal vertex $T$ contained in $T_v$, then let $e'$ denote the first edge on the unique path in $T$ from $w$ to $u$. 
Let $(A'\mid B')$ be the corresponding cut, and assume without loss of generality that $B'\subset B$.
We set $f(w):=B'$.
\end{enumerate}
Injectivity of $f$ is a consequence of the fact that $P_T$ satisfies the axiom (P3). We claim that $f$ is also surjective.
Let $x$ be a vertex of $V_G$. If $x=\{i\}$ is a singleton, then $f(i)=x$. If $x=A_1$ is a cluster contained in $A$, then
let $e_1$ be the edge corresponding to the cut $(A_1\mid N\setminus A_1)$. Then $f$ maps one of the two vertices of $e_1$ to $x$.
Similarly, we can find a preimage if $x=A_1$ is a cluster contained in $B$.
Therefore, $f$ provides a bijection between $V$ and $V_G$. 

We claim that $f$ is a graph isomorphism. Let $v_1,v_2\in V$. If $v_1,v_2$ are both leaves, then $\{v_1,v_2\}$ is not an edge of $T$
and $\{f(v_1),f(v_2)\}$ is not an edge of $E_G$.

Assume $v_1\in N$ and $v_2\not\in N$. Then $\{v_1,v_2\}$ is an edge of $T$ if and only if
$v_2$ is the unique vertex adjacent to $v_1$, and this is true if and only if $f(v_2)$ is the unique
minimal cluster contained in $A$ and containing $v_1$, and this is true if and only if 
$\{f(v_1),f(v_2)\}$ is an edge in $E_G$. The case $v_1\in B$ is treated analogously.

If $v_1$ and $v_2$ are non-leaves of $T$ contained in $T_u$. Assume that $e'=\{v_1,v_2\}$ is an edge of $T$. 
Let $(A'\mid B')$ be the corresponding cut; without loss of generality, assume $A'\subset A$ and
$A'\cup \{v_1\}$ is the set of leaves of one of the two components that we get when we remove $e'$ from $T$.
Let $\{v_2,v_3\}$ be the first edge on the path from $v_2$ to $u$ and let $(A''\mid B'')$ be its corresponding cut 
such that $A''\subset A$. Then we see that $f(v_1)=A'$ and $f(v_2)=A''$. Also we know that $A'\subset A''\subset A$
and there is no other cluster of $C_T$ in between $A''$ and $A'$ with respect to inclusion.
This implies that $\{f(v_1),f(v_2)\}$ is an edge in $E_G$. Conversely, if $\{f(v_1),f(v_2)\}$ is an edge in $E_G$,
then two edges corresponding to the cuts $(f(v_1)\mid N\setminus f(v_1))$ and $(f(v_2)\mid N\setminus f(v_2))$
have to equal, and the corresponding edge is the edge $\{v_1,v_2\}$ in $T$. The case where both $v_1$ and $v_2$ 
are non-leaves of $T$ contained in $T_v$ is similar.

If $v_1\in T_u$ and $v_2\in T_v$, then $\{v_1,v_2\}$ is an edge of $T$ if and only if $v_1=u$ and $v_2=v$,
and this is true if and only if $\{f(v_1),f(v_2)\}$ is an edge in $E_G$.
\end{proof}

\begin{corollary}
 Let $T$ be a phylogenetic tree. For any two cuts $c_1,c_2\in C_T$, the cut graphs $G_{C_T,c_1}$ 
 and $G_{C_T,c_2}$ and $T$ are all isomorphic. 
\end{corollary}
\begin{proof}
 Immediate consequence of Proposition \ref{prop:T2C}.
\end{proof}

\begin{lemma} 
Let $C$ be a phylogenetic set of cuts of $N$. Let $c=(A\mid B)\in C$ be a cut. Then $C$ is equal to $C_{G_{C,c}}$.
\end{lemma}

\begin{proof}
Denote the last added edge in the construction of $G_{C,c}$ as $e=\{v,u\}$ and assume without loss of generality that $A$ 
is the leaf set of component $T_v$ when we remove edge $e$ from $G_{C,c}$. Let $c'=(A'\mid B')$ be an arbitrary cut in $C$. 
If $c'=c$, then it is equal to the cut in $C_{G_{C,c}}$ corresponding to the edge $e$, hence $c'\in C_{G_{C,c}}$.

Assume $c'\ne c$. Because $C$ is phylogenetic, we know that exactly one of the following four statements 
$A'\subset A$, $A'\subset B$, $B'\subset A$, $B'\subset B$ is true. Without loss of generality, assume that $A'\subset A$.
Then $A'\in V(A)$. Let $w$ be the first vertex on the unique path from $A'$ to $u$. Then we see that 
the corresponding cut for edge $\{A', w\}$ is $(A'\mid B')$. Hence $(A'\mid B')\in C_{G_{C,c}}$.
Because $c'$ was chosen arbitrarily, we conclude that $C\subset C_{G_{C,c}}$. 

Now, take any cut $c'=(A'\mid B')\in C_{G_{C,c}}$. If $c'$ is the cut corresponding to the edge $e$, then $c=\{A\mid B\}$,
which implies $c\in C$. Assume $c'$ corresponds to some edge $e'=\{v',u'\}$ in $T_v$. Without loss of generality,
we may suppose that $u'$ is on the unique path from $v'$ to $u$. Then $c'=(v'\mid N\setminus v')$. Since $v'$ is a cluster
of $C$, we obtain that $c'\in C$. If  $c'$ corresponds to some edge in $T_u$, we proceed analogously. Therefore we
get $C_{G_{C,c}}\subset C$ and consequently $C= C_{G_{C,c}}$.
\end{proof}

For any phylogenetic set of cuts $C$, we can choose a cut $c\in C$. The class of $G_{C,c}$ does not depend on the choice
of $c$, so this construction provides an inverse to $t_{TC}:{\cal T}_N\to {\cal C}_N$.

\subsection{Cuts and Crossings}

\begin{proposition}
For every phylogenetic set $C$ of cuts, the crossing relation $X_C$ is phylogenetic.
\end{proposition}

\begin{proof}
Assume that $i,j,k,l,m$ are pairwise distinct (but otherwise arbitrary) elements of $N$. Assume, for the sake of contradiction, that $(i,j\mid k,l)$ and
$(i,k\mid j,l)$ are both in $X_C$. Then there is a cut $(A_1\mid B_1)$ such that $i,j\in A_1$ and $k,l\in B_1$, and another
cut $(A_2\mid B_2)$ such that $i,k\in A_2$ and $j,l\in B_2$. Then all four sets $A_1\cap A_2$, $A_1\cap B_2$, $B_1\cap A_2$,
and $B_2\cap B_2$ are not empty. This contradicts Axiom~(C). Hence the assumption must have been wrong, which proves
that Axiom~(X1) is fulfilled.

Now assume that $(i,j\mid k,l)$ and $(i,j\mid l,m)$ are both in $X_C$. 
By Axiom~(C), the set of all clusters that contain both $i$ and $j$ and that do not contain $l$ is totally ordered by
inclusion. Let $A$ be the smallest such cluster. Then $A$ does not contain $k$ and does not contain $m$. Then $(A\mid N\setminus A)$
is a cut with $i,j$ on the left side and $k,m$ in the right side. Hence $(i,j\mid k,m)$ is in $X_C$, and it follows
that Axiom~(X2) is fulfilled.

Now assume that $(i,j\mid k,l)$ is in $X_C$. Then there is a cut $(A\mid B)$ such that $i,j\in A$ and $k,l\in B$. If $m\in A$,
then $(i,m\mid k,l)$ is in $X_C$, and if $m\in B$, then $(i,j\mid l,m)$ is in $X_C$. It follows that Axiom~(X3) is fulfilled,
and that $X_C$ is phylogenetic.
\end{proof}

The following definition is only needed for the proof of Lemma~\ref{lem:partial}.

\begin{definition}
A {\em partial cut} of $N$ is a cut of some subset of $N$. 
Fix a phylogenetic crossing relation $X$. 
We say that a cut or a partial cut $(A\mid B)$ is {\em compatible} with $X$ if and only if for any distinct $i,j\in A$
and distinct $k,l\in B$, there is a crossing relation $(i,j\mid k,l)\in X$.
\end{definition}

\begin{example} \label{ex:ibase}
If $(i,j\mid k,l)\in X$, then the cross itself, considered as a partial cut, is compatible with $X$.
\end{example}

\begin{lemma} \label{lem:partial}
Let $X$ be a phylogenetic crossing relation on $N$. Then, for any $(i,j\mid k,l)\in X$. 
Then there exists a cut $(A\mid B)$ compatible with $X$ such that $i,j\in A$ and $k,l\in B$.
\end{lemma}

\begin{proof}
We prove that for any $n$ such that $4\le n\le |N|$, there is a partial cut $(A\mid B)$ 
compatible with $X$ such that $i,j\in A$ and $k,l\in B$,
and $|A|+|B|=n$. We proceed by induction on $n$. For $n=4$, the statement is trivially true. 

Assume $5\le n\le |N|$. By induction hypothesis, there exists a compatible partial cut $(C\mid D)$ such that 
$i,j\in C$, $k,l\in D$, and $|C|+|D|=n-1$. 
Let $m\in N\setminus (C\cup D)$. We claim that either $(C\cup \{m\}\mid D)$ or $(C\mid D\cup\{m\})$ is compatible with $X$.
Assume, for the sake of contradiction, that this claim is wrong. Then there exist $a,b,p\in C$ and $c,r,s\in D$ such that $a\ne b$, $r\ne s$, and
the relations $(a,b\mid c,m)$ and $(p,m\mid r,s)$ do not hold. We may also assume $a\ne p$ and $r\ne c$.

Since $(C\mid D)$ is compatible, it follows that $(a,p\mid r,s)$ holds. By Axiom~(X3), it follows that
$(a,p\mid r,m)$ or $(p,m\mid r,s)$ holds --- but we have that $(p,m\mid r,s)$ does not hold, hence we have $(a,p\mid r,m)$.
If $b=p$, then $(a,b\mid r,m)$ holds. If $b\ne p$, then we use
$(b,p\mid r,s)$ and Axiom~(X3) and get $(b,p\mid r,m)$, since $(p,m\mid r,s)$ does not hold. Then, from $(a,p\mid r,m)$ and 
$(b,p\mid r,m)$, we obtain $(a,b\mid r,m)$ by Axiom~(X2).
By the compatibility of $(C\mid D)$, we get $(a,b\mid r,c)$.
By Axiom~(X2), from $(a,b\mid r,c)$ and $(a,b\mid r,m)$, we get $(a,b\mid c,m)$. This contradicts the assumption.
\end{proof}

We can now the define a transformation $t_{XC}$ from sets of cuts to crossing relations.

\begin{definition}
For any phylogenetic crossing relation $X$ on $N$,
we define $t_{XC}(X)$ as the set of all cuts that are compatible with $X$ and 
denote it as $C_X$.
\end{definition}

\begin{proposition}
For any phylogenetic crossing relation $X$, the set $C_X$ of compatible cuts is phylogenetic.
\end{proposition}

\begin{proof}
Assume, for the sake of contradiction, that $C_X$ does not fulfill Axiom~(C). Then there exists cuts $(A_1\mid B_1)$, $(A_2\mid B_2)$ in $C_X$
and four elements $i\in A_1\cap A_2$, $j\in A_1\cap B_2$, $k\in B_1\cap A_2$, and $l\in B_1\cap B_2$. 
Because $(A_1\mid B_1)$ is compatible, we have $(i,j\mid k,l)\in X$.
Because $(A_2\mid B_2)$ is compatible, we have $(i,k\mid j,l)\in X$.
This contradicts Axiom~(X1).
\end{proof}

\begin{theorem}
The two sets ${\cal X}_N$ and ${\cal C}_N$ are in bijection:
function $t_{XC}:{\cal X}_N\to {\cal C}_N$ is the inverse of function $t_{CX}:{\cal C}_N\to {\cal X}_N$.
\end{theorem}

\begin{proof}
Let $X\in {\cal X}_N$ and $i,j,k,l\in N$ pairwise distinct. If $(i,j\mid k,l)$ is in $X$, then Lemma~\ref{lem:partial}
implies that there is a cut $(A\mid B)$ in $C_X$ such that $i,j\in A$ and $k,l\in B$. Therefore $(i,j,k,l)$ is also in $X_{C_X}$.
Conversely, if $(i,j\mid k,l)$ is in $X_{C_X}$, then there is a cut $(A\mid B)$ in $C_X$ such that $i,j\in A$ and $k,l\in B$.
Since $(A\mid B)$ is compatible with $X$, it follows that $(i,j\mid k,l)$ is in $X$.

Let $C\in {\cal C}_N$ be a phylogenetic set of cuts. Let $(A\mid B)$ be a cut. If $(A\mid B)$ is in $C$, then all quadruples $(i,j\mid k,l)$ with
$i,j\in A$ and $k,l\in B$ are in $X_C$. Hence $(A\mid B)$ is also in $C_{X_C}$. Conversely, assume that $(A\mid B)$ is in $C_{X_C}$.
We pick elements $a\in A$ and $b\in B$. Let $\alpha$ be the set of clusters of $C$ that contain $a$ but not $b$,
and let $\beta$ be the set of clusters of $C$ that contain $b$ but not $a$.
By Axiom~(C), both sets $\alpha$ and $\beta$ are totally ordered by set inclusion.
For any $i\in A\setminus\{a\}$ and $j\in B\setminus\{b\}$, the quadruple $(a,i\mid b,j)$ is in $X_C$ because $(A\mid B)$ is compatible with $X_C$. 
Then, there must exist a cut $(A'\mid B')\in C$ such that $a,i\in A'$ and $b,j\in B'$. Consequently, we have a cluster $A'\in\alpha$
for every element $i\in A\setminus\{a\}$, and therefore the largest cluster of $\alpha$ is a superset of $A$. Similarly, we can
show that the largest cluster of $\beta$ is a superset of $B$.

Let $A''$ be the smallest cluster of $\alpha$ that is still a superset of $A$.
We claim that $A''=A$;
if this claim is true, then $(A\mid B)$ would be in $C$, which would finish the proof of the converse and of the whole theorem.

To prove the claim, we assume, for the sake of contradiction, that there is an element $c\in A''\setminus A=A''\cap B$. Let $\alpha'$ be the set
of clusters of $C$ that contain $a$ but neither $b$ nor $c$. This set is also totally ordered by set inclusion. 
For any choice of elements
$k\in A\setminus\{a\}$ and $l\in B\setminus\{b\}$, the quadruple $(a,k\mid b,l)$ is in $X_C$. Hence there exists a cut
in $(A'''\mid B''')$ of $C$ such that $a,k\in A'''$ and $b,l\in B'''$. In particular, $A'''$ is in $\alpha'$. Since we can vary
$k$, it follows that the largest cluster of $\alpha'$ is a superset of $A$. Hence there is a superset of $A$ in $\alpha$ that does not
contain $c$, which is a contradiction to the fact that the smallest cluster of $\alpha$ that contains $A$, namely $A''$, does
contain $c$.
\end{proof}

\subsection{Partitions and Equivalences}

In order to prepare for the conversion between partitions and equivalences, we prove a result which could be considered as a kind of converse
of Lemma~\label{lem:tree}. Let us say that a partition {\em separates} a triple $\{a,b,c\}\in {N\choose 3}$ if and only if $a$, $b$, and $c$
are in three pairwise distinct parts of the partition. 

\begin{theorem} \label{thm:pisp}
Let $P$ be a collection of partitions of $N$ satisfying (P1) such that for every triple in ${N\choose 3}$, there is a unique partition separating it.
Then $P$ is phylogenetic.
\end{theorem}

\begin{remark} \label{rem:converse}
The converse is also true: if a collection of partition is phylogenetic, then it is equal to $P_T$ for some phylogenetic tree, 
by Theorem~\ref{thm:hh}. By Lemma~\ref{lem:unique_vertex}, it follows that every triple is separated by a unique partition in $P_T$.
\end{remark}

In order to prove Theorem~\ref{thm:pisp}, we need the following proposition.

\begin{lemma} \label{lem:unique_pair}
Let $P$ be a collection of partitions of $N$ satisfying Axiom (P1) such that for every triple in ${N\choose 3}$, there is a unique partition separating it.
Let $p_1,p_2$ be two distinct partitions. Then there is a unique pair of parts $A_1\in p_1$ and $A_2\in p_2$
such that $A_1\cup A_2=N$. 

Moreover, if $B_1\in p_1$ is any part of $p_1$ distinct from $A_1$, and $B_2\in p_2$ is any part of $p_2$ distinct from $A_2$,
then $B_1\subset A_2$, $B_2\subset A_1$, and $B_1\cap B_2 = \emptyset$.
\end{lemma}

\begin{proof}
For $i=1,2$, 
let $a_i,b_i,c_i\in N$ be elements from three different parts of $p_i$. The partition separating $\{a_i,b_i,c_i\}$ is unique,
therefore at least two of $a_1,b_1,c_1$ must be in the same part of $p_2$; without loss of generality, we may assume 
that $a_1$ and $b_1$ are in the same part.
We choose $A_2$ to be this part. Similarly, we may assume that $a_2$ and $b_2$ are in the same part of $p_1$, and we choose $A_1$ to be 
this part.

Suppose that $A_1\cup A_2\subsetneq N$. Take any $x\in N\setminus (A_1\cup A_2)$. Assume, without loss of generality, that $x$ is not in the same part 
with $a_1$ in $p_1$ --- otherwise, we exchange $a_1$ and $b_1$. Analogously we may assume that $x$ is not in the same part with 
$a_2$ in $p_2$. 
Then we see that the triple $\{x,a_1,a_2\}$ is separated by
both partitions $p_1$ and $p_2$. We have our contradiction. It follows that $A_1\cup A_2=N$.

The second statement is a consequence of $A_1\cup A_2=N$ and the fact that the any part of $p_i$ different from $A_i$ is a subset of 
$N\setminus A_i$, for $i=1,2$.
\end{proof}

\begin{remark} \label{rem:C}
As a consequence of Lemma~\ref{lem:unique_pair}, the set of parts of any set of $P$ fulfills an axiom that is similar to the cluster axiom (C): any two parts,
whether they are in the same partition or not, are either contained one in the other, or disjoint, or their union is $N$. 
\end{remark}

\begin{proof}[Proof of Theorem~\ref{thm:pisp}] 
We already know that $P$ satisfies Axiom~(P1), by assumption.

Axiom (P2): let $a\in N$ be arbitrary. Let $b\in B$ such that $b\ne a$. By Remark~\ref{rem:C}, the set of all parts of any partition that contain $a$
but not $b$ is totally ordered by inclusion. Let $A$ be the smallest such part. We claim that $A=\{a\}$. Assume, for the sake of contradiction,
that $A$ contains a second element $c\ne a$. Then the triple $\{a,b,c\}$ cannot be separated by any partition, contradicting the assumption.

Axiom (P3): for the sake of contradiction, assume that $A$ is a part with $|A|\geq 2$ that shows up in two distinct partitions $p_1,p_2\in P$.
By Lemma \ref{lem:unique_pair}, there exists $B\in p_2$ such that $A\cup B=N$. This implies that $|p_2|=2$,
which violates Axiom (P1).

Axiom (P4): Suppose, for the sake of contradiction, that (P4) is not fulfilled.
Let $A$ be a part of some partition of $p$ of $P$, with cardinality at least 2, such that $N\setminus A$ is not a part 
of any partition.
Let $a$ and $b$ be two distinct elements of $A$ .
Let ${\cal X}$ be the set of all parts that are supersets of $N\setminus A$ and do not contain $b$.

To show that ${\cal X}$ is not empty, we pick an element $e\in N\setminus A$.
There must be a partition $q$ separating the triple $\{a,b,e\}$. By Lemma~\ref{lem:unique_pair}, there are parts $F\in p$
and $E\in q$ such that $E\cup F=N$. The part $A\in p$ has non-empty intersection with at least two parts of $q$, namely with the part containing $a$
and with the part containing $b$. Neither part can be a superset of $A$. Then, by Remark~\ref{rem:C}, both are subsets of $A$.
This implies $F=A$ and therefore $E\cup A=N$. Then we get $e\in E$ and $b\not\in E$. So, $E\in{\cal X}$.

By Remark~\ref{rem:C}, ${\cal X}$ is totally ordered by inclusion.  
Let $C$ be the smallest element of ${\cal X}$. Since $N\setminus A$ is not a part, 
there exists an element $d$ in $C\setminus (N\setminus A)=A\cap C$. 
Now we repeat the argument that we used above with the triple $\{d,b,e\}$.
Note that $d$ and $b$ are two distinct elements in $A$, since $d$ is in some part
of $\cal{X}$ while $b$ is not. 
Hence there must be a partition $q'$ separating the triple $\{d,b,e\}$. By Lemma~\ref{lem:unique_pair},
there are parts $F'\in p$ and $E'\in q'$ such that $F'\cup E'=N$. Then, with the analogous reasoning, 
we obtain that the parts containing $d$ and the part containing $b$ in $q'$ are both subsets of 
$A$. Hence we have that $F'=A$ and $E'\cup A=N$. Also we observe that $b\notin E'$. Therefore,
$E'\in \cal{X}$ and $d\notin E'$.

Since ${\cal X}$ is totally ordered by inclusion and $C$ is the smallest element of $\cal{X}$,
we have that $C\subset E'$. This implies $d\in E'$, which is a contradiction.

\end{proof}

\begin{remark}
It is a fun question to ask what happens to if we replace ``triples'' by ``quadruples'', ``quintuples'' etc. The second author conjectures
that there are almost no collections of partitions such that each partition has four parts, and every quadruple is separated by a unique partition.
More precisely, any such collection has only a single partition, where every part is a singleton.
\end{remark}

In order to convert triples to partitions, we also need one more definition.

\begin{definition}
For any partition $p$ of $N$, let $S_p$ be the set of all triples that are separated by $p$.

For any set $S$ of triples, let $G_S$ be the graph with vertex set $N$, and an edge between $i,j\in N$ if and only if
no triples of $S$ contain
both $i$ and $j$. Let $p_S$ be the partition of $N$ defined by the connected components of $G_S$.
\end{definition}

\begin{example}
 In Figure \ref{fig:partition_graph} we see the graphs $G_S$ when $S$ is one of the three equivalence classes of triples
 in Example \ref{ex:equivsmall}.
 \begin{figure}
   \begin{center}
    \begin{tikzpicture}[scale=0.8]
       \node[lnode] at (0,0) (1) {1};
       \node[lnode] at (1,0.8) (2) {2};
       \node[lnode] at (2,0) (3) {3};
       \node[lnode] at (0.5,-1) (4) {4};
       \node[lnode] at (1.5,-1) (5) {5};
       \draw[bedge] (3)edge(4) (3)edge(5) (4)edge(5);
    \end{tikzpicture}
    \hspace{2cm}
    \begin{tikzpicture}[scale=0.8]
       \node[lnode] at (0,0) (1) {1};
       \node[lnode] at (1,0.8) (2) {2};
       \node[lnode] at (2,0) (3) {3};
       \node[lnode] at (0.5,-1) (4) {4};
       \node[lnode] at (1.5,-1) (5) {5};
       \draw[bedge] (1)edge(2) (1)edge(3) (2)edge(3);
    \end{tikzpicture}
    \hspace{2cm}
    \begin{tikzpicture}[scale=0.8]
       \node[lnode] at (0,0) (1) {1};
       \node[lnode] at (1,0.8) (2) {2};
       \node[lnode] at (2,0) (3) {3};
       \node[lnode] at (0.5,-1) (4) {4};
       \node[lnode] at (1.5,-1) (5) {5};
       \draw[bedge] (1)edge(2) (4)edge(5);
    \end{tikzpicture}

\end{center}
\caption{These are the graph $G_S$ defined by the three equivalence classes in Example \ref{ex:equivsmall}. 
	Each of the graphs is a disconnected union of complete graphs, in accordance to Lemma~\ref{lem:complete_graph}.}
\label{fig:partition_graph}
\end{figure}
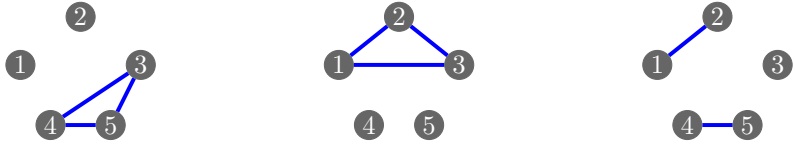
\end{example}

\begin{example}
Let $p=p_a$ be as in Example~\ref{ex:parts}. Then $S_{p_a}$ is exactly the equivalence class of triples $E_a$ in Example~\ref{ex:tb}.
Moreover, $p_{E_a}$ is the partition $p_a$.
\end{example}

\begin{lemma}\label{lem:complete_graph}
Let $S$ be a diverse set of triples. Then $G_S$ is a disconnected union of complete graphs.
 
\end{lemma}
\begin{proof}
 Let $i,j,k\in N$ be three distinct vertices of $G_S$. Suppose
that $\{i,j\}$, $\{j,k\}$ are edges of $G_S$, but $\{i,k\}$ is not
an edge. Then there exists $l\in N\setminus \{i,j,k\}$ such that $\{i,k,l\}\in S$. By Axiom (D1), 
one of the triples $\{i,j,l\}$, $\{j,k,l\}$, or
$\{i,j,k\}$ is in $S$. This violates the fact that $\{i,j\}$ and $\{j,k\}$ are edges. This shows that
any two vertices in the same connected
component of $G_U$ are connected by an edge; the graph is a disconnected union of complete graphs.
\end{proof}

\begin{lemma} \label{lem:div2part}
For any partition $p$ with at least three parts, we have that $S_p$ is diverse and $p_{S_p}=p$.

For any diverse set $S$ of triples, the partition $p_S$ has at least three parts, and we have $S_{p_S}=S$.
\end{lemma}

\begin{proof}
Let $p$ be a partition with at least three parts. Then we know that $S_p$ is non-empty.
Assume that the triple $\{i,j,k\}$ is in $S_p$, which means that $i$, $j$, and $k$ are in three
distinct parts. A fourth element $l\in N$ can at most be in one of this three parts, hence $p$ separates $l$ 
and two other elements out of $i$, $j$, and $k$.
Therefore $S_p$ satisfies (D1). 

Let $a,b,c,x,y,z\in N$. Assume that $p$ separates the triples $\{ a,x,y\}$, $\{ b,y,z\}$, and $\{ c,x,z\}$. 
Then $x$, $y$, and $z$ are in pairwise
distinct parts, so $p$ also separates $\{x,y,z\}$. Therefore $S_p$ satisfies (D2), hence $S_p$ is diverse.
Moreover,
 $i,j$ are in the same part of $p$ if and only if no triples in $S_p$
contain both $i$ and $j$ if and only if $i$ and $j$ are in the same component of $G_{S_p}$ if and only
if $i$ and $j$ are in the same part of $p_{S_p}$. Hence $p_{S_p}=p$.

Let $S$ be any diverse set of triples. 
%
%
%
If a triple $\{i,j,k\}$ is in $S$, then
none of $\{i,j\}$, $\{i,k\}$ or $\{j,k\}$ is an edge of $G_S$. By Lemma \ref{lem:complete_graph},
we obtain that $i,j,k$
are in pairwise distinct components in $G_S$. Hence, $i,j,k$ are in pairwise
distinct parts of $p_S$. Therefore, $\{i,j,k\}\in S_{p_S}$. For the other direction, let 
$\{i,j,k\}$ be any triple in $S_{p_S}$. This indicates that $i,j,k$ are in three 
pairwise distinct parts in $p_S$, i.e., $i,j,k$ are in three pairwise distinct components
of the graph $G_S$. Therefore, none of $\{i,j\}$, $\{i,k\}$, $\{j,k\}$ is an edge of $G_S$.
Hence, $S$ contains triples $\{i,j,a\}$, $\{i,k,b\}$ and $\{j,k,c\}$ for some $a,b,c\in N$.
By Axiom (D2), $\{i,j,k\}\in S$. Hence $S_{p_S}=S$.

\end{proof}

Now we define the conversion from equivalences on triples to collections of partitions. 
For any phylogenetic equivalence relation $E$ of triples in $N$, we define $P_E$ as the collection of all partitions
$p_U$ for any equivalence class $U$ of $E$. The function from ${\cal E}_N$ to $\overline{{\cal P}_N}$ that
maps $E$ to $P_E$ is denoted by $t_{EP}$.

\begin{lemma}
Let $E$ be a phylogenetic equivalence relation of triples in $N$. Then
$P_E$ is a phylogenetic set of partitions.
\end{lemma}

\begin{proof}
Every equivalence class $U$ contains at least one triple. This triple is separated by $p_U$, and it follows
that $p_U$ must have at least three parts. Therefore $P_E$ satisfies Axiom~(P1). 

Every triple $\tau=\{i,j,k\}\in {N\choose 3}$ is in a unique equivalence class $U$, and $p_U$ separates $\tau$
by Lemma~\ref{lem:div2part}. Moreover, if $V$ is any equivalence class such that $p_V$ separates $\tau$, 
then $\tau\in S_{p_V}=V$, and therefore $U=V$. This implies that every triple is separated by a unique
partition in $P_E$. By Theorem \ref{thm:pisp}, $P_E$ is phylogenetic.
\end{proof}

For any phylogenetic collection $P$ of partitions, each triple $\tau\in {N\choose 3}$ is 
separated by a unique partition in $P$, by Remark~\ref{rem:converse}.
We define $E_P$ as follows: Two triples $\tau_1$ and $\tau_2$ are equivalent if and only if the unique partition
separating $\tau_1$ is the same as the unique partition separating $\tau_2$. The function from ${\cal P}_N$ to
${\cal E}_N$ that maps $P$ to $E_P$ is denoted by $t_{PE}$.
It is straightforward to see that $t_{TE}=t_{PE}\circ t_{TP}$.

\begin{lemma}
 Let $P$ be a phylogenetic set of partitions of $N$. Then $E_P$ is a phylogenetic equivalence relation of triples in ${N\choose 3}$.
\end{lemma}

\begin{proof}
For any partition $p$ in $P$, the set of triples $S_p$ is diverse by Lemma~\ref{lem:div2part}; hence Axiom~(E0) is fulfilled.

\end{proof}

\begin{theorem} 
The two sets $\cal{P}_N$ and $\cal{E}_N$ are in bijection: function $t_{EP}:\cal{E}_N\to \cal{P}_N$, $E\mapsto P_E$
is the inverse of function $t_{PE}:\cal{P}_N\to \cal{E}_N$, $P\mapsto E_P$.
\end{theorem}

\begin{proof}
For any $E\in \cal{E}_N$, any class $U$ of $E$ is diverse. Then, by Lemma \ref{lem:div2part}
we have that $S_{p_U}=U$. We see that $\{S_{p_U}\}_{U\in E}$ is exactly $E_{P_E}$ --- where the foot index $U\in E$
means that $U$ is a class of $E$. Hence, we have $E_{P_E}=E$.

For any $P\in \cal{P}_N$, each partition $p\in P$ has at least three parts. By Lemma~\ref{lem:div2part}, 
we know that $p_{S_p}=p$. Since $\{p_{S_p}\}_{p\in P}$ is exactly the partition collection $P_{E_P}$,
we obtain that $P_{E_P}=P$.

\end{proof}

In Figure~\ref{fig:diag2}, we display the diagram consisting of all conversion algorithms in this paper.

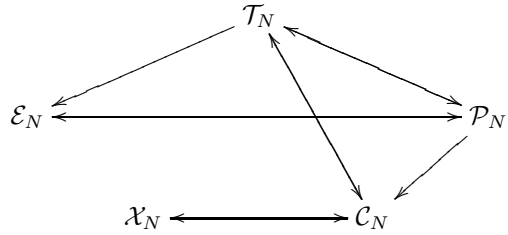
\begin{figure}
\[ \xymatrix{
  & & {\cal T}_N \ar[lld] \ar[rrd] \ar[rdd] & & \\
  {\cal E}_N \ar[rrrr] & & & & {\cal P}_N \ar[ld] \ar[ull] \ar[llll] \\ 
  & {\cal X}_N \ar[rr] & & {\cal C}_N \ar[ll] \ar[uul] &
} \]
\caption{This diagram shows all conversion maps between different types of structures that have been defined 
in this paper. We also have seen that the triangles are commutative.}
\label{fig:diag2}
\end{figure}

\section*{Acknowledgement}
This research was funded by the Austrian Science Fund (FWF): W1214-N15, project DK9, and by (FWF): P31338.

The authors thank Nicolas Allen Smoot for the inspiring discussion on the conversion from a phylogenetic
set of cuts to a phylogenetic tree.
The authors thank Antonio Jim\'enez-Pastor for the discussion on the conversion
 from a crossing relation to an equivalence relation on the triples.

 \newpage
\bibliographystyle{plain}
\bibliography{reps}

\end{document}